\documentclass[12pt]{article}

\usepackage{amsfonts}
\usepackage{amssymb}
\usepackage{amsthm}
\usepackage{amsmath}
\usepackage{amscd}
\usepackage{mathrsfs}
\usepackage{enumerate}
\usepackage{ulem}
\usepackage{fontenc}
\usepackage[all]{xy}
\usepackage{array}
\usepackage[english]{babel}

\def\CC {{\mathbb C}}     
\def\LL {{\mathbb L}}     
\def\NN {{\mathbb N}}     
\def\OO {{\mathbb O}}     
\def\PP {{\mathbb P}}     
\def\XX {{\mathbb X}}     
\def\ZZ {{\mathbb Z}}     

\def\lw  {\longrightarrow}
\def\mc {\mathcal}
\def\mk {\mathfrak}

\def\ol  {\overline}

\def\rw  {\rightarrow}

\def\ul  {\underline}

\newtheorem{theorem}{Theorem}[section]
\newtheorem{lemma}[theorem]{Lemma}

\newtheorem{prop}[theorem]{Proposition}

\newtheorem{coro}[theorem]{Corollary}
\newtheorem{rem}{Remark}[section]
\newtheorem{defin}{Definition}[section]

\newtheorem{ex}{Example}[section]

\begin{document}

\title{Secant varieties and degrees of invariants}

\author{Valdemar V. Tsanov}

\maketitle

\begin{abstract}
The ring of invariant polynomials $\CC[V]^G$ over a given finite dimensional representation space $V$ of a complex reductive group $G$ is known, by a famous theorem of Hilbert, to be finitely generated. The general proof being nonconstructive, the generators and their degrees have remained a subject of interest. In this article we determine certain divisors of the degrees of the generators. Also, for irreducible representations, we provide lower bounds for the degrees, determined by the geometric properties of the unique closed projective $G$-orbit $\XX$, and more specifically its secant varieties. For a particular class of representations, where the secant varieties are especially well behaved, we exhibit an exact correspondence between the generating invariants and the secant varieties intersecting
the semistable locus. \footnote{This work was supported by DFG grant Sachbeihilfe DFG-AZ: TS 352/1-1.}
\end{abstract}

\small{
\tableofcontents
}


\section{Introduction}

Let $G$ be a reductive complex algebraic group and $V$ be a finite dimensional $G$-module. A famous theorem of Hilbert asserts that the ring of invariant polynomials $\CC[V]^G$ is finitely generated. The general proof being nonconstructive, the generators and their degrees have remained a topic of significant interest. The generators can be chosen homogeneous and, while the finite set of generators is not unique, as at least scaling is always possible, their degrees are uniquely determined, if one convenes to an increasing order and takes multiplicity into account. We say that $\CC[V]^G$ admits a generator of degree $d>0$, if the degree component $\CC[V]_d^G$ is not contained in the subring generated by $\CC[V]_{<d}^{G}$. The maximal degree of a generator is called the Noether number, ${\rm No}(G,V)$, this is the minimal $d$ for which $\CC[V]_{\leq d}^G$ generates $\CC[V]^G$. If $\CC[V]^G\ne\CC$, we denote by $d_1$ the minimal positive degree of a generator.

In the following theorem we identify certain divisors of the degrees of the generators. The proof is given in Section \ref{Sect RDsetsAndInv}.

\begin{theorem}\label{Theo I RDSiml-Invar}
Let $H\subset G$ be a Cartan subgroup with weight lattice $\Lambda$ and root system $\Delta\subset\Lambda$. Let $\Lambda(V)$ be the set of $H$-weights of a given finite dimensional $G$-module $V$. Suppose that $M\subset \Lambda(V)$ satisfies the following two properties:
\begin{enumerate}
\item[{\rm (i)}] $M\cap (M+\Delta)=\emptyset$ (we call such sets root-distinct);
\item[{\rm (ii)}] $M$ is linearly dependent over $\ZZ_{>0}$ and minimal with this property (we call such sets balanced simplices, and we denote by $b_M=\sum\limits_{\nu\in M} b_\nu)$, where $b_\nu\in\ZZ_{>0}$ are the unique coefficients with greatest common divisor 1 such that $\sum\limits_{\nu\in M} b_\nu\nu=0$.)
\end{enumerate}
Then the ring of invariants $\CC[V]^G$ admits a generator of degree $kb_M$ for some integer $k\geq 1$.
\end{theorem}

\begin{rem}
The above notion of a root-distinct set generalizes a notion introduced by Wildberger, \cite{Wildberger}, in his study of momentum maps. Wildberger's approach was taken further by Sjamaar, \cite{Sjamaar}, Smirnov, \cite{Smirnov}, and Hristova, Ma\c{c}ia\.{z}ek and the author, \cite{HMT}. All these works seek to understand momentum images or, equivalently, Brion polytopes in Smirnov's case. In the present note we relate the root-distinct sets to degrees of invariants. A similar construction of invariants corresponding to sets of weights has been introduced by Wehlau, \cite{Wehlau93}, but limited to the case of tori, where the root-distinctness is clearly irrelevant. Wehlau's condition is indeed an analogue to (ii). Thus, the root-distinctness property given in condition (i) is, in a sense, the supplementary assumption allowing to generalise Wehlau's construction from tori to general reductive groups. 
\end{rem}

Let $\PP=\PP(V)$ be the projective space of $V$ with the induced group action. We denote by $J=\CC[V]_{\geq 1}^G$ the ideal in the invariant ring vanishing at $0$. Let $\PP^{us}\subset\PP$ be the zero-locus of $J$, also known as the unstable locus, or the nullcone. The complement $\PP^{ss}=\PP\setminus\PP^{us}$ is called the semistable locus. We denote by $\sqrt{I(J)}$ the radical ideal of $\CC[V]$ generated by $J$.

In the second part of the article, we focus on the case $G$ is semisimple and $V$ is irreducible and nontrivial, so that $V=V(\lambda)$, where $\lambda\ne 0$ is the highest weight with respect to any fixed Borel subgroup $B\subset G$. There is a unique closed $G$-orbit in the projective space, the orbit of highest weight vectors, which we denote by
$$
\XX=G[v_\lambda]\subset \PP \;.
$$
For $r\in\NN$, the secant variety $\sigma_r(\XX)$ is defined as the Zariski close of the union of linear spaces spanned by $r$ points on $\XX$:
$$
\Sigma_r=\sigma_r(\XX)=\ol{\bigcup\limits_{x_1,...,x_r\in\XX} \PP Span(x_1,...,x_r)} \subset \PP \;.
$$
Since the representation is irreducible, $\XX$ spans the projective space. Thus, for large enough $r$, we get $\sigma_r(\XX)=\PP$; the minimal such $r$ is called the generic rank in $\PP$ with respect to $\XX$, denoted here by $r_g$.

The secant varieties behave naturally well with respect to the group action, but their definitions are independent of the group, once the variety $\XX$ is given. Thus the properties of these varieties, their ideals and coordinate rings could be expected to have nontrivial consequences on the properties of invariant rings. The closed orbit is always unstable $\XX \subset \PP^{us}$. As a corollary of a result of Landsberg and Manivel, \cite{Lands-Mani-2003-ProjGeo}, we obtain the following theorem.

\begin{theorem}
\begin{enumerate}
\item[{\rm (i)}] Let $r_{us}=\max\{r\in\NN: \sigma_r(\XX)\subset \PP^{us}\}$. Then $r_{us}<d_1$.
\item[{\rm (ii)}] If $\PP^{us}\subset \sigma_r(\XX)\ne \PP$ for some $r$, then $\sqrt{I(J)}_{r+1}$ contains the $(r-1)$-st prolongation of the degree 2 component of the ideal of $\XX$, i.e.,
$$
I_2(\XX)^{(r-1)}=(I_2(\XX)\otimes S^{r-1}V^*)\cap S^{r+1}V^* \subset \sqrt{I(J)}_{r+1}.
$$
\end{enumerate}
\end{theorem}

The proof is presented in Section \ref{Sect SecAndInv} along with some basic material about secant varieties. In Section \ref{Sect rssemicont}, we show that for one class of varieties with particularly well behaved secant varieties, the generators of the invariant ring correspond exactly to the secant varieties intersecting the semisitable locus, and in particular, there is a bijective correspondence between the sets of numbers $\{r_{us}+1,...,r_g\}$ and $\{d_1,...,{\rm No}(G,V)\}$. Remarkably, in several cases these sets are in fact equal.

\section{Root-distinct sets of weights and invariants}\label{Sect RDsetsAndInv}

In this section we give a proof of the main existence theorem.

\begin{proof}[Proof of Theorem \ref{Theo I RDSiml-Invar}]
Let $T\subset H$ be the maximal compact subgroup of the given torus $H\subset G$, and let $K\subset G$ be a maximal compact subgroup containing $T$. Let $\langle,\rangle$ be a $K$-invariant Hermitean form on $V$. We consider the Killing form on ${\mk k}$ and use it to set an isomorphism between ${\mk k}$ and ${\mk k}^*$, and to embed ${\mk t}^*$ as a subspace of ${\mk k}^*$. 

We consider the $K$-equivariant momentum map:
$$
\mu : \PP \lw i{\mk k}^* \;,\quad \mu[v](\xi)=\frac{\langle\xi v,v\rangle}{\langle v,v\rangle} \;,\; [v]\in\PP,\xi\in{\mk k}\;.
$$
We denote the fibres of the momentum by ${\mc M}_\xi=\mu^{-1}(\xi)$. Then ${\mc M}_0$, if non-empty, is preserved by $K$, contained in the semistable locus and, by the theorem of Kirwan ${\mc M}_0/K = \PP_{ss}//G$, where the latter denotes the GIT-quotient.

The weight space decomposition $V=\oplus_{\nu\in\Lambda(V)}V_\nu$ is orthogonal. For any subset $M\subset \Lambda(V)$, we denote $V_M=\oplus_{\nu\in M}V_\nu$ and $\PP_M=\PP(V_M)\subset \PP$.

\begin{lemma}(Wildberger)
 
If $M\subset\Lambda(V)$ is a root-distinct set, then $\mu(\PP_M)={\rm Conv}(M)\subset{\mk t}^*$.
\end{lemma}

In fact, Wildberger, \cite{Wildberger}, proved a similar statement for some very special sets of weights, namely Weyl group orbits, but his method goes through in the general case, as we show below.

\begin{proof} Let $\mk g=\mk h \oplus (\oplus_{\alpha\in\Delta}\mk g_{\alpha})$ be the weight space decomposition with respect to $T$ and let $\Delta^+$ be the system of positive roots for a fixed Borel subgroup $B\subset G$ containing $T$. It is well known that the root vectors $e_\alpha\in\mk g_{\alpha}$ can be chosen so that $\mk k$ is spanned by $\mk t$ and $e_\alpha-e_{-\alpha},i(e_\alpha+e_{-\alpha})$ for $\alpha\in\Delta^+$. The defining formula for $\mu$ clearly extends to a map $\PP\to \mk g^*$ and we denote the coordinate functions, for $\xi\in\mk g$, by $\mu^{\xi}:\PP\to\CC$, $\mu^{\xi}[v]=\langle \xi v,v\rangle/\langle v,v\rangle$. It is easy to see that the simultaneous vanishing of $\mu^{e_\alpha-e_{-\alpha}}$ and $\mu^{i(e_\alpha+e_{-\alpha})}$ is equivalent to the vanishing of $\mu^{e_\alpha}$ and $\mu^{e_{-\alpha}}$. Hence $\mu[v]\in i\mk t^*$ if and only if $\mu^{e_\alpha}[v]=0$ for all $\alpha$. 

The root vectors send weight spaces to weight spaces: $e_\alpha(V_\nu)\subset V_{\nu+\alpha}$. For a given nonzero $v\in V$, the orthogonal projections to the weight spaces, ${\rm pr}_\nu:V\to V_\nu$, define a unique decomposition as a sum of weight vectors $v=\sum {\rm pr}_\nu(v)$; the set ${\rm St}(v)=\{\nu\in\Lambda(V):{\rm pr}_\nu(v)\ne 0\}$ is called the support of $v$. Now observe that ${\rm St}(v)$ is a root-distinct set if and only if ${\rm St}(v)\cap{\rm St}(e_\alpha v)=\emptyset$ for all $\alpha\in\Delta$. In such a case, the orthogonality of weight spaces implies $\mu^{e_\alpha}[v]=0$ for all $\alpha$, and by the above remarks we get $\mu[v]\in i{\mk t}^*$. In fact $\mu[v]=\mu_T[v]$, where $\mu_T$ denotes the momentum map for the $T$-action, which is equal to the composition of $\mu$ with the orthogonal projection from $i\mk k^*$ to $i\mk t^*$. We can conclude that, for a given root-distinct set $M\in\Lambda(V)$, we have 
$$
\mu(\PP_M)=\mu_T(\PP_M)\;,
$$
The latter is known to equal ${\rm Conv}(M)$ by the well known theorem of Atiyah for momentum maps of tori. However, for the case at hand the direct calculation of $\mu_T[v]$ is also easily accessible. We may assume that $||v||=1$. For $\nu\in{\rm St}(v)$, put $a_\nu=||{\rm pr}_\nu(v)||$ and $v_\nu=\frac{1}{a_\nu}{\rm pr}_\nu(v)$. Then
$$
\mu_T[v]=\mu_T[\sum\limits_{\nu\in{\rm St}(v)} a_\nu v_\nu ] = \sum\limits_{\nu\in M} |a_\nu|^2 \nu \quad,\quad {\rm with}\quad \sum\limits_{\nu\in M} |a_\nu|^2=1 \;.
$$
This implies $\mu(\PP_M)={\rm Conv}(M)\subset{\mk t}^*$.
\end{proof}

Returning to the proof of the theorem, the hypothesis (i), in view of the above lemma, implies that $\mu(\PP_M)={\rm Conv}(M)$. Form hypothesis (ii) we infer that $0\in{\rm Conv}(M)$, and hence ${\mc M}_0\cap\PP_M\ne\emptyset$. In particular, picking any set of weight vectors $v_\nu\in V_{\nu}$ of norm 1, and setting
$$
v=\sum\limits_{\nu\in M} \sqrt{b_\nu}v_\nu\;, \quad\textrm{we get}\quad \mu[v]=\frac{1}{||v||}\sum\limits_{\nu\in M}b_\nu\nu =0\;.
$$
By Heckman's theorem, \cite{Heckman82}, it follows that $[v]\notin \PP^{us}$ and hence there exists a nonconstant homogeneous invariant polynomial $f\in \CC[V]^G$ with $f(v)\ne 0$. The restriction of $f$ to the torus orbit $Hv\subset V$ is a nonzero constant. The orbit closure $\LL=\ol{Hv}\subset\PP_M$ is a linear subspace of $\PP$ and so, in particular, a projective toric $H$-variety. Let $R$ denote the homogeneous coordinate ring of $\LL$ and let ${\rm res}:\CC[V]\rw R$ denote the quotient morphism, which is surjective and $H$-equivariant. We have ${\rm res}(\CC[V]^G)\subset R^H$; this restricted map is not necessarily surjective, but we have ${\rm res}(f)\ne0$. Since $\LL$ is a toric variety, $R^H$ is isomorphic to a polynomial ring on one variable. The hypothesis (ii) and specifically the supplementary assumption ${\rm gcd}\{b_\nu:\nu\in M\}=1$ implies that $R^H$ is generated by ${\rm res}(p)$, where
$$
p=\prod\limits_{\nu\in M} z_\nu^{b_\nu} \quad,\quad {\rm deg}(p)=b_M \;,
$$
where $z_\nu$ denotes the coordinate to $v_\nu$. Indeed, $p$ is $H$-invariant, it does not vanish on $v$, it is not a power of any other polynomial, and there is no $H$-invariant monomial of smaller degree in the variables $z_\nu,\nu\in M$.

We have ${\rm res}(f)\ne 0$, hence ${\rm res}(f)={\rm res}(p)^k$ for some $k$, and ${\rm deg}(f)=k b_M$.
\end{proof}

\begin{ex}\label{Example PiQ rd balancedsimplex} (The adjoint representation of a simple group)

Consider the case where $G$ is simple of rank $\ell$ and $V=\mk g$ is the adjoint representation. It is well known that $\CC[\mk g]^G$ is isomorphic to a polynomial ring in $\ell$ variables. The degrees $d_1,...,d_\ell$ of the generators are also well known, and are related to a variety of important objects associated to $G$. The minimal degree is $2$ and the maximal is the Coxeter number $h$ of $G$. 

The set of weights with respect to a Cartan subgroup $H$ is $\Delta\cup\{0\}$. Let $\Pi$ be the set of simple roots with respect to a fixed Borel subgroup $B$ and let $\theta=\sum\limits_{\alpha\in\Pi}m_\alpha \alpha$ be the highest root, which is also the highest weight of $\mk g$. The Coxeter number, which we also denote by $h_\Pi$ when necessary, is the given by
$$
h=h_\Pi=1+\sum\limits_{\alpha\in\Pi}m_\alpha \;.
$$
Denote $\Pi^{Q}=\Pi\cup\{-\theta\}$. This set is a root-distinct balanced simplex with
$$
b_{\Pi^Q}=1+\sum\limits_{\alpha\in\Pi}m_\alpha=h_\Pi \;.
$$
More generally, consider any subset $\tilde\Pi\subset \Pi$ corresponding to a simple subgroup of $G$, i.e., to a connected subdiagram of the Dynkin diagram. We denote $\tilde\Pi^{Q}=\tilde\Pi\cup\{-\tilde\theta\}$, where $\tilde\theta$ is the highest root of the sub-root-system generated by $\tilde\Pi$. Then $\tilde\Pi^Q$ is a root-distinct balanced simplex in $\Lambda$, and $b_{\tilde\Pi^Q}=h_{\tilde\Pi}$ is the Coxeter number of the corresponding simple root system. Hence there is an invariant generator of degree $\tilde qh_{\tilde\Pi}$ for some $\tilde q\in\NN$. A connected Dynkin diagram with $\ell$ nodes admits connected subdiagrams of with $\tilde\ell$ nodes for any $1\leq \tilde\ell\leq \ell$. Since $\#\tilde\Pi^Q=\#\tilde\Pi+1$, the root system $\Delta$ admits root-distinct balanced simplices of all dimensions from 1 to $\ell$. Using the known values Coxeter number and the degrees of generating invariants, one finds out that this procedure yields all generators. An a priori proof of this fact could be of interest.
\end{ex}

\section{Secant varieties and degrees of invariants}\label{Sect SecAndInv}

We assume now that $V=V(\lambda)$ is an irreducible representation of a semisimple group $G$, with $\lambda\in\Lambda^+$ being the highest weight of $V$ with respect to a fixed pair of Cartan and Borel subgroups $H\subset B\subset G$. We let $\XX=G[v_\lambda]\subset\PP=\PP(V)$ be the orbit of the highest weight line, the unique closed projective $G$-orbit. We also assume that $\lambda\ne 0$, so that the representation is nontrivial.

We consider linear combinations of points from the affine cone over $\XX$. Since the representation is irreducible the variety $\XX$ spans the ambient space. The rank function on $\PP$ with respect to $\XX$ is defined as
$$
{\rm rk}_\XX:\PP\lw\NN \;,\quad {\rm rk}_\XX[v]=\min\{r\in\NN:v=x_1+...+x_r, [x_j]\in\XX\} \;.
$$
The rank subsets of $\PP$ are defined as
$$
\XX_r = \{[v]\in\PP:{\rm rk}_\XX[v]=r\} \;.
$$
The $r$-th secant variety of $\XX$ is defined as the Zariski closure of the union of linear spaces spanned on $r$ points of $\XX$:
$$
\Sigma_r=\Sigma_r(\XX) = \ol{\bigcup\limits_{x_1,...,x_r\in\XX} \PP Span(x_1,...,x_r)} =\ol{\bigcup\limits_{s\leq r}\XX_s}\;.
$$
The border rank on $\PP$ with respect to $\XX$ is defined as
$$
\ul{\rm rk}_\XX :\PP\le\NN \;,\quad \ul{\rm rk}_\XX[v]=\min\{r\in\NN:[v]\in\Sigma_r \}\;.
$$
There is a unique integer $r_g\geq1$ for which $\XX_{r_g}$ is open in $\PP$ and this is the smallest $r$ for which $\Sigma_r=\PP$; $r_g$ is called the generic rank. We also denote by $r_{\rm max}$ the maximal value of ${\rm rk}_\XX$. The rank function is $G$-invariant, and consequently the sets $\XX_r$ and $\Sigma_r$ are preserved by $G$. We have containments of varieties $\XX\subset\Sigma_2\subset...\subset\Sigma_{r_g}=\PP$ and there is a corresponding chain of $G$-stable ideals $I(\XX)\supset I(\XX_2)\supset...\supset 0$. The ideal of $\XX$ is generated in degree 2, by a suitable generalization of the Pl\"ucker relations, due to Kostant, see e.g. \cite{Landsberg-2012-book}. The following theorem describes the first nonzero degree component of the ideal of the $r$-th secant variety of a variety cut out by quadrics. We formulate it here for the case in hand. We use the identification of $\CC[V]$ with the space of symmetric tensors $SV^*$ inside the tensor algebra on $V^*$.

\begin{theorem}{\rm (Landsberg and Manivel, \cite{Lands-Mani-2003-ProjGeo})}
 
The first nonzero homogeneous component of the ideal $I(\Sigma_r)$ is in degree $r+1$ and is given, for $r\geq 2$, by the $(r-1)$-st prolongation of the generating space $I_2(\XX)$ of the ideal of $\XX$, i.e.
$$
I_r(\Sigma_r)=0 \quad , \quad I_{r+1}(\Sigma_r)=S^{r+1}V^* \cap (I_2(\XX)\otimes S^{r-1}V^*) \;.
$$
\end{theorem}

\begin{defin} The rank of instability of the irreducible representation $V=V(\lambda)$ is defined as
$$
r_{us}= \max\{r\in\NN : \Sigma_r\subset \PP^{us} \} \; \;.
$$
When $\CC[V]^G\ne\CC$, so that the semistable locus is nonempty, the rank of semistability is defined as 
$$
r_{ss}= \max\{r\in\NN : \Sigma_r\subset \PP^{ss}\ne\emptyset \} = r_{us}+1 \; \;.
$$
\end{defin}

\begin{rem}

1) The closed $G$-orbit $\XX\subset\PP$ belongs to $\PP^{us}$ as long as the representation is nontrivial; thus $r_{us}\geq1$.

2) The incidence with the nullcone for a projective variety can be tested via the momentum map with respect to a maximal compact subgroup $K\subset G$ and an invariant Hermitean form on $V$. We have: $\Sigma_r\subset\PP^{us}$ if and only if $0\notin \mu(\Sigma_r)$.
\end{rem}

The following proposition is an interpretation of a result of Zak, \cite{Zak-Book}, Ch. III.

\begin{prop}
If $V(\lambda)\ncong V(\lambda)^*$, then $\Sigma_2\subset\PP^{us}$. If $\CC[V(\lambda)]^G\ne\CC$, then $r_{ss}=2$ if and only if $V(\lambda)\cong V(\lambda)^*$.
\end{prop}

The above results have the following direct consequence.

\begin{theorem}\label{Prop Sec and d-Inv}

If a nonconstant homogeneous invariant $f\in\CC[V(\lambda)]^G$ vanishes on $\Sigma_r$, then $deg(f)>r$.

Suppose that $\CC[V(\lambda)]^G\ne \CC$ and let $d_1$ be the minimal positive degree of a homogeneous invariant polynomial. Then
$$
r_{ss}\leq d_1 \;.
$$
\end{theorem}

In view of the above theorem, it is natural to ask: given $\Sigma_r\cap\PP^{ss}\ne 0$, is there indeed an invariant of degree $r$? Such an invariant may or may not appear (see example 1.1) but the relation between rank and degree is not accidental. It stems from the fact that certain monomials of invariant polynomials have the form $(z_1...z_r)^k$, where $z_j$ are coordinates with respect to vectors $x_1,...,x_r$ in the affine cone $\hat\XX$, taken as a part of a basis in $V(\lambda)$. This is formulated below after introducing some notation.

Let $W=N_G(H)/H$ denote the Weyl group. Given $\lambda\in\Lambda^+$, the weights $w\lambda, w\in W$ are called the extreme weights of the irreducible $G$-module $V(\lambda)$. The corresponding projective points are exactly the $H$-fixed points in $\XX=G[v_\lambda]$: 
$$
\XX^H = \{x_w=[v_{w\lambda}]\;:\; w\in W\} \;.
$$
This set is in a one-to-one correspondence with the coset space $W/W_\lambda$.



\begin{coro}
Let $\lambda\in\Lambda^+$ and $\XX=G[v_\lambda]\subset\PP(V(\lambda))$. If the Weyl group orbit $W\lambda$ contains a root-distinct balanced simplex with $r$-elements, then $\CC[V(\lambda)]^G$ admits a generator which does not vanish on the secant variety $\Sigma_r(\XX)$ and $r\geq r_{ss}$.
\end{coro}

\begin{proof}
The hypothesis means that there exist $w_1,...,w_r\in W$ satisfying the following two conditions:

(i) The set of weights is root-distinct, i.e. $w_j\lambda-w_k\lambda\notin\Delta$.

(ii) $w_1\lambda,...,w_r\lambda$ form a balanced simplex.

The construction in the above theorem yields a generator $f$ of the ring of invariants $\CC[V(\lambda)]^G$, whose restriction to the secant space $\PP Span(v_{w_1\lambda},...,v_{w_r\lambda})$ is the monomial $(z_1^{b_1}\dots z_r^{b_r})^q$, where $q$ is a positive integer, $z_j$ are the coordinates associated to the weight vectors $v_{w_j\lambda}$, and $b_1,...,b_r$ is the unique set of positive integers with greatest common divisor 1 such that $\sum b_jw_j\lambda=0$. Now, clearly $f$ does not vanish on $\PP Span(v_{w_1\lambda},...,v_{w_r\lambda})\subset \Sigma_{r}(\XX)$.
\end{proof}

\begin{ex}
Let us consider the case $\lambda=k\rho$, where $\rho=\frac12\sum\limits_{\alpha\in\Delta^+}\alpha$. For $k\geq 2$, we have $\mu(\PP)=\mu(\Sigma_2)$.
\end{ex}

\begin{ex} (Veronese varieties)

Consider $G=SL_n$ acting on $V=S^k\CC^n$ with $k,n\geq2$. The associated homogeneous projective variety is the Veronese variety $\XX={\rm Ver}_k(\PP^{n-1})\subset\PP(V)$. The extreme weight vectors for a given Cartan subgroup $H$ the are $k$-th powers $v_j^k$ of the corresponding basis vectors $v_1,...,v_n\in\CC^n$; the $W$-orbit of the highest weight $\lambda=k\omega_1$ forms a balanced $n-1$-simplex centered at $0$, with $b_{W\lambda}=\# W\lambda=n$. This simplex is root-distinct, as $k\geq 2$. (For $k=1$, the natural representation of $SL_n$ has no root-distinct sets of weights with more than one element, which amounts to the fact that the projective space is a single $SL_n$-orbit.) Hence $\CC[S^k\CC^n]^{SL_n}$ admits a generator of degree $qn$ for some $q\in\NN$. It is not hard to see that, for $1\leq r\leq n$, $G$ has an open orbit in the secant variety $\Sigma_r$ and the momentum image $\mu(\Sigma_r)\cap i{\mk t}^*$ is the $r$-skeleton of the simplex ${\rm Conv}(W\lambda)$. Thus $\Sigma_r\subset\PP^{us}$ for $r<n$ and $r_{ss}=n$. We can conclude that $d_1\geq n$, where $d_1$ is the minimal degree of a nonconstant homogeneous polynomial in $\CC[S^k\CC^n]^{SL_n}$. For $k=2$, the determinant of symmetric matrices is an invariant of degree $n$. The case $n=3$ shows that, for large $k$, invariants in degree $n$ may or may not occur.
\end{ex}

\begin{rem}
The rank function ${\rm rk}_\XX$ on $\PP$ is independent of any group actions on $\XX$, but the notions of rank of instability and semistability depend on the group. We use the notation $r_{us,G},r_{ss,G}$, when the group needs to be specified. Certain homogeneous projective varieties admit transitive actions by proper subgroups of their automorphism group, say $\tilde G\subset G=Aut\XX$. Such transitive actions of subgroups have been classified by Onishchik. For simple $G$, $\tilde G$ is also simple, and the cases include $\PP^{2n-1}$, homogeneous under $SL_{2n}$ and $Sp_{2n}$; the varieties of pure spinors in the irreducible spin-representation, which are the same for $SO_{2n}$ and $SO_{2n-1}$, the quadric $Q^5$, homogeneous under $SO_7$ and $G_2$. In general, a homogeneous projective variety has a semisimple automorphism group and splits as a product of homogeneous varieties $\XX=\XX^1\times \dots\XX^k$ corresponding to the simple factors of $G$. Then a transitive subgroup $\tilde G\subset G$ is necessarily a product $\tilde G=\tilde G^1\times\dots\times \tilde G^k$, with $\tilde G^j$ transitive on $\XX^k$.

Let us consider the behaviour of the rank for (infinitesimally) simple $G$ and a proper subgroup $G$ acting transitively on $\XX$. It turns out that $\tilde G$ is a simple group without outer automorphisms, i.e., all of its representations are self dual. Thus $r_{us,\tilde G}=1$. Whenever semistable points exist (the only exception the transitive action on the entire projective space $\XX=\PP^{2n-1}$, by $\tilde G=Sp_{2n}\subset SL_{2n}=G$, where we have $r_{max}=1$) the rank of semistability for the subgroup is $r_{ss,\tilde G}=2$ and the representation admits an $\tilde G$-invariant of even degree.
\end{rem}






\section{Rank-semi-continuous varieties}\label{Sect rssemicont}

The correspondence between secant varieties intersecting the semistable locus and generating invariants turns out to be exact for a class of homogeneous projective varieties introduced and classified in \cite{Petukh-Tsan}, where the rank function behaves particularly well.

\begin{defin}
A projective variety $\XX\subset \PP$ is called rank-semicontinuous, or rs-continuous, if the rank and border rank functions on $\PP$ with respect to $\XX$ coincide: ${\rm rk}_\XX=\ul{\rm rk}_\XX$, i.e., points of higher rank cannot be approximated and the Zariski closure in the definition of secant varieties is not necessary.
\end{defin}

\begin{theorem}(\cite{Petukh-Tsan})
A homogeneous projective variety $\XX\subset \PP$ is rs-continuous if and only if $\Sigma_2=\XX\cup \XX_2$, i.e., the required property holds for all $r$ if and only if it holds for $r=2$. The classification of the homogeneous rs-continuous varieties is given in the table below.
\end{theorem}

The table below contains the list of rs-continuous homogeneous varieties with their linear automorphism groups, along with the valies of the rank function on the semistable locus, whenever the latter is nonempty. In the last column we give the degrees $d_1,...,d_k$ of a minimal generating set of homogeneous elements of the invariant ring $\CC[V(\lambda)]^G$. The degrees can be found in Kac's table \cite{Kac-nilp-orb}. (All rs-continuous varieties turn out to have polynomial invariant rings.) The nullcones are also known in all cases, see e.g. in Zak, \cite{Zak-Book}, Ch. 3. We put $d_1=0$ if $\CC[V]^G=\CC$; this is reflected as $r_{us}=r_{max}$ in the rank column.

Note that most of the ambient spaces are spaces of matrices or bilinear forms and the rank notion has a classical interpretation.

\begin{center}
{\bf rs-continuous varieties, semistable range of rank, degrees of invariants}\\
\vspace{0.5cm}
\begin{tabular}{|c|c|c|c|}
\hline
Variety $\XX\subset \PP(V)$ & $G=Aut\XX$ &  $r_{ss},...,r_{max}$ & $d_1,..d_k$ \\
\hline
\hline
$\PP(\CC^n) = \PP(\CC^n)$& $SL_n$ & $r_{us}=r_{max}=1$ & 0\\
\hline
${\rm Ver}_2(\PP(\CC^n))\subset \PP({\rm S}^2\CC^n)$ & $SL_n$ & $r_{ss}=r_{max}=n$ & $n$ \\
\hline
${\rm Gr}_2(\CC^n) \subset \PP(\Lambda^2\CC^n)$ & $SL_n$ & $r_{us}=r_{max}=\lfloor\frac{n}{2}\rfloor$ & 0 for odd $n$; \\ && $r_{ss}=r_{max}=\lfloor\frac{n}{2}\rfloor$ & $n/2$ for even $n$\\
\hline
${\rm Fl}_{1,n-1}(\CC^n)\subset \PP(\mk{sl}_n)$ & $SL_n$ & $2,...,n$ & $2,...,n$ \\
\hline
Q$^{n-2} \subset \PP(\CC^n)$ & $SO_n$ & $2$ & $2$ \\
\hline
S$^{10} \subset \PP^{15}=\PP(\Lambda^{even}\CC^{5})$ & $Spin_{10}$ & $r_{us}=r_{max}=2$ & 0 \\
\hline
Gr$_2(\CC^{2n},\omega) \subset \PP(\Lambda_0^2\CC^{2n})$ & $Sp_{2n}$ & 2,...,n & 2,4,6...,2n-2\\
\hline
E$^{16} \subset \PP^{26}=\PP({\rm Herm}_{3\times 3}(\OO)_\CC)$ & $E_6$ & $3$ & $3$ \\
\hline
F$^{15} \subset \PP^{25}=\PP({\rm SHerm}_{3\times 3}(\OO)_\CC)$ & $F_4$ & $2,3$ & $2,3$ \\
\hline
${\rm Seg}(\PP^{m-1}\times\PP^{n-1}) \subset \PP(\CC^m\otimes\CC^n)$ & $SL_m\times SL_n$ & $r_{us}=r_{max}=\min\{m,n\}$ & 0 for $m\ne n$; \\ && $r_{ss}=r_{max}=m$ & $m$ for $m=n$ \\
\hline
\end{tabular}
\end{center}

The notation in the above table is as follows: ${\rm Ver}_2$ denotes the quadratic Veronese embedding; ${\rm Gr}_2$ is the Grassmannian of 2-planes; ${\rm Fl}_{1,n-1}$ denotes the 2-step flag variety of lines contained in hyperplanes; ${\rm Q}^{n-2}$ is the $(n-2)$-dimensional quadratic hypersurface; ${\rm S}^{10}$ is the 10-dimensional variety of pure spinors; ${\rm Gr}_2(\CC^{2n},\omega)$ is the variety of 2-planes in $\CC^{2n}$ isotropic for a given nondegenerate skew-symmetric form $\omega$; ${\rm Herm}_{3\times3}(\OO)_\CC$ denotes the complexified space of octonionic Hermitean $3\times 3$-matrices and ${\rm SHerm}_{3\times3}(\OO)_\CC$ is the subspace defined by vanishing of the trace; ${\rm E}^{16}$ is the set of such matrices with rank 1, which can be defined using the Freudenthal determinant and ${\rm F}^{15}={\rm E}^{16}\cap\PP({\rm SHerm}_{3\times3}(\OO)_\CC)$; ${\rm Seg}$ denotes the Segre embedding of a product of projective spaces.

We obtain the following.

\begin{theorem}
Suppose that $\XX\subset\PP$ is rs-continuous and $G$ is the linear automorphism group of $\XX$. Then the invariant ring is polynomial $\CC[V_\lambda]^G=\CC[f_1,...f_k]$. The degrees of the generators, ordered nonincreasingly, coincide with the values of the rank function on the semistable locus: 
$$\{r_{ss},...,r_{max}\}=\{d_1,...,d_k\}$$
except in the case of $Sp_{2n}$ acting on ${\rm Gr}_2(\CC^{2n},\omega)\subset\PP(\Lambda_0^2\CC^{2n})$, where the two sets are related by the bijection $d_j=2(r_j-1)$. Furthermore, the generators can be chosen to vanish on the secant varieties as follows:
$$
\Sigma_r\subset Z(f_{k-(r_{max}-r-1)},...,f_k) \;,\; for\;\; r_{us}\leq r\leq r_{max}-1\;.
$$
\end{theorem}

\begin{rem}

1) The rs-continuous varieties, where either $r_{us}=r_{max}$ or $r_{ss}=r_{max}$ (so that by the theorem $\CC[V_\lambda]^G$ has either zero or one generator) are exactly the so-called subcominuscule varieties: the closed projective orbits in the isotropy representations of irreducible Hermitian symmetric spaces. Incidentally, these varieties are obtained as closed projective orbits in irreducible representations of reductive groups where the action on the projective space is spherical, but the spherically acting group can be a subgroup of the automorphism group.

2) In the remaining three cases, where $r_{ss}<r_{max}$ (so that $\CC[V_\lambda]^G$ has more than one generator), the variety $X$ is a hyperplane section in an rs-continuous variety $\tilde X$ of the first type, and the rank function is obtained by restriction; these are ${\rm Fl}_{1,n-1}(\CC^n)={\rm Segre}(\PP^{n-1}\times(\PP^{n-1})^*)\cap\PP(\mk{sl}_{n})$, ${\rm Gr}_2(\CC^{2n},\omega)={\rm Gr}_2(\CC^{2n})\cap\PP(\omega^\perp)$ and ${\rm F}^{15}={\rm E}^{16}\cap\PP({\rm SHerm}_{3\times 3}(\OO)_\CC)$. The corresponding embeddings of linear automorphism groups are symmetric subgroups, i.e., given by fixed point sets of involutions $SL_n\subset SL_n\times SL_n$, $Sp_{2n}\subset SL_{2n}$ and $F_4\subset E_6$. 
\end{rem}

\bibliographystyle{plain}

\small{

}

\vspace{0.4cm}

\author{\noindent Valdemar V. Tsanov \\
Ruhr-Universit\"at Bochum\\
Fakult\"at f\"ur Mathematik, IB 3/101\\ 
D-44780 Bochum, Germany.\\
Email: Valdemar.Tsanov@rub.de}\\

\end{document}